\documentclass[12pt, twoside, leqno]{article}

\usepackage[english]{babel}
\usepackage{abstract} 
\usepackage[latin1]{inputenc}
\usepackage{verbatim}
\usepackage[arrow, matrix, curve]{xy}
\usepackage{amsmath,amssymb,amsfonts,amsthm,amsopn}
\usepackage{nicefrac}

\DeclareMathOperator{\ord}{ord}
\DeclareMathOperator{\PrePer}{PrePer}

\newtheorem{Theorem}{Theorem}[section]

\newtheorem{corollary}[Theorem]{Corollary}
\newtheorem{Proposition}[Theorem]{Proposition}
\newtheorem{Lemma}[Theorem]{Lemma}

\theoremstyle{definition}
\newtheorem*{Definition}{Definition}

\newtheorem{Example}[Theorem]{Example}
\newtheorem{Remark}[Theorem]{Remark}

\numberwithin{equation}{section}

\begin{document}

\baselineskip=17pt

\title{Heights of points with bounded ramification}
\author{Lukas Pottmeyer}
\date{\today}

\maketitle

\renewcommand{\thefootnote}{}


\renewcommand{\thefootnote}{\arabic{footnote}}
\setcounter{footnote}{0}

\selectlanguage{english}
\begin{abstract}
Let $E$ be an elliptic curve defined over a number field $K$ with fixed non-archimedean absolute value $v$ of split-multiplicative reduction, and let $f$ be an associated Latt\`es map. Baker proved in \cite{Ba03} that the N\'eron-Tate height on $E$ is either zero or bounded from below by a positive constant, for all points of bounded ramification over $v$. In this paper we make this bound effective and prove an analogue result for the canonical height associated to $f$. We also study variations of this result by changing the reduction type of $E$ at $v$. This will lead to examples of fields $F$ such that the N\'eron-Tate height on non-torsion points in $E(F)$ is bounded from below by a positive constant and the height associated to $f$ gets arbitrarily small on $F$.
\end{abstract}

\begin{section}{Introduction}\label{Introduction}

By Kronecker's Theorem the standard logarithmic height $h$ of an algebraic number vanishes precisely at zero and roots of unity. The height $h$ is a non-negative function and so one can ask whether the height of non roots of unity can get arbitrarily small in the set of algebraic numbers. The classical example to see that this is the case is the sequence $\{2^{\nicefrac{1}{n}}\}_{n\in\mathbb{N}}$, because we have $h(2^{\nicefrac{1}{n}})=\frac{1}{n}\log{2}$. Northcott's famous theorem from 1950 (see \cite{No50}) implies that an upper bound on the degree of an algebraic non root of unity yields a lower bound for its height. An interesting question is which other properties of an algebraic number yield a lower bound of its height. 
Of course, one can ask the same question for all kinds of height functions. If $E$ is an elliptic curve defined over the algebraic numbers, then we denote by $\widehat{h}_E$ the N\'eron-Tate height on $E$. Moreover, for a rational function $f\in \overline{\mathbb{Q}}(x)$ of degree at least $2$, we denote by $\widehat{h}_f$ the canonical height associated to $f$. A point is called preperiodic if its forward orbit is finite. The set of all preperiodic points of $f$ will be denoted by $\PrePer(f)$. We will adopt a notation from Bombieri and Zannier (see \cite{BZ01}) to these heights.

\begin{Definition}\rm
Let $L$ be a subfield of $\overline{\mathbb{Q}}$ and $f\in\overline{\mathbb{Q}}(x)$ with $\deg(f)\geq2$. We say $L$ \textit{has the Bogomolov property relative to} $\widehat{h}_f$ if and only if there exists a constant $c>0$ such that $\widehat{h}_f (\alpha)\geq c$ for all $\alpha \in L\setminus \PrePer(f)$.

If $E$ is an elliptic curve defined over $\overline{\mathbb{Q}}$, we have a similar definition. We say $L$ \textit{has the Bogomolov property relative to} $\widehat{h}_E$ if and only if there exists a constant $c>0$ such that $\widehat{h}_E (P)\geq c$ for all $P \in E(L)\setminus E_{\rm tor}$.
\end{Definition}

As we have $\widehat{h}_{x^2} =h$, the case $f=x^2$ yields the definition of Bombieri and Zannier.
Let $\mathbb{Q}^{tr}$ be the maximal totally real subfield of $\overline{\mathbb{Q}}$. Furthermore, let $K$ be a number field with non-archimedean place $v\in M_K$. Denote by $K^{ab}$ the maximal abelian field extension of $K$, and by $K^{nr,v}$ the maximal algebraic extension of $K$ which is unramified above $v$. Now we can summarize the known examples of fields with the Bogomolov property relative to $\widehat{h}_E$ in Table \ref{Bogomolov elliptic}. Of course, such examples may depend on the choice of the elliptic curve $E$.

\begin{table}[ht]
	\centering
		\begin{tabular}{| c | c | c |}
		 \hline
			Field & Restrictions & Reference \\ \hline\hline
			finite extensions of $\mathbb{Q}^{tr}$ & none & Zhang \cite{Zh98} \\ \hline
			$K^{ab}$ & $E / K$ & Baker, Silverman \cite{Ba03},\cite{Si04} \\ \hline
			totally $p$-adic of any type & $p\geq 3$ & Baker, Petsche \cite{BP05} \\ \hline
			finite extensions of $K^{nr,v}$ & $E / K_v$ a Tate curve & Baker \cite{Ba03}\\ \hline
			$\mathbb{Q}(E_{\rm tor})$ & $E/\mathbb{Q}$ & Habegger \cite{Ha11} \\ \hline
		\end{tabular}
	\caption{Fields with the Bogomolov property relative to $\widehat{h}_E$}
	\label{Bogomolov elliptic}
\end{table}

For the definition of a totally $p$-adic field we refer to Section \ref{Heights and Lattes maps}. The result of Baker is not stated explicitly, it can be found in \cite{Ba03}, Section 5, Case 1. Notice that the result of Baker and Petsche is effective. In their paper they also prove an effective lower bound for $\widehat{h}_E$ on $\mathbb{Q}^{tr}$, whenever $E$ is defined over $\mathbb{Q}^{tr}$. The results of Baker (see \cite{Gu07}), Zhang, and Baker and Silverman (see \cite{BS04}) are also true in the setting of abelian varieties. The generalization of a Tate curve in case of finite extensions of $K^{nr,v}$ is an abelian variety which is totally degenerate at $v$. In this paper we will prove an effective lower bound for Baker's result.

For an elliptic curve $E$, defined over a field $K$ with non-archimedean absolute value $v$, and any $e\in \mathbb{N}$ denote by $M_{e}^{E}(v)$ the set of points $P$ such that the ramification index $e_{w\mid v}$ is bounded by $e$ for all $w\mid v$ in $M_{K(P)}$.

\begin{Theorem}
Let $E$ be an elliptic curve defined over a number field $K$ with split-multiplicative reduction at a finite place $v$ on $K$. Then there are effective computable constants $c'$, $c'_T >0$, depending on the degree of $K$, $e$, $v$ and the $j$-invariant of $E$, such that $\widehat{h}_E (P) \geq c'$ for all $P \in M_{e}^{E}(v) \setminus E_{\rm tor}$, and such that there are less than $c'_T$ torsion points in $M_{e}^{E}(v)$.
\end{Theorem}

Northcott's motivation for his theorem was to prove that an endomorphism of an algebraic variety has only finitely many preperiodic points of bounded degree. Here the degree of a point is the degree of the smallest number field over which the point is defined. A strong relation between elliptic curves and dynamical systems in dimension one is given by Latt\`es maps. Let $f$ be a Latt\`es map associated to the elliptic curve $E$.
It is easy to see that the Bogomolov property of a field $F$ relative to $\widehat{h}_f$ implies the Bogomolov property of $F$ relative to $\widehat{h}_E$. The converse, however, is not true in general as we will see in Example \ref{counterexample elliptic}. This example shows also that the Bogomolov property relative to $\widehat{h}_E$ is not preserved under finite field extensions, a fact that is also known for the standard logarithmic height $h$. (The only known counterexample is the field extension $\mathbb{Q}^{tr}(i)/\mathbb{Q}^{tr}$. See \cite{AN07} for the first proof of this, \cite{ADZ11} for a very short proof and \cite{Po13} for a proof using dynamical heights). Hence, it is worth studying fields with the Bogomolov property relative to $\widehat{h}_f$. In Proposition \ref{twisting} we present a condition when a lower bound for $\widehat{h}_E$ can be transfered to a lower bound for $\widehat{h}_f$. In complete analogy to the definition of $M_{e}^{E}(v)$ above we define $M_{e}(v)$ as the set of algebraic numbers $\alpha$ such that the ramification index $e_{w\mid v}$ is bounded by $e$ for all $w\mid v$ in $M_{K(\alpha)}$.

\begin{Theorem}
Let $E$ be an elliptic curve defined over a number field $K$, with split-multiplicative reduction at a finite place $v$ on $K$. Further let $f$ be a Latt\`es map associated to $E$. Then there are effective computable constants $c$, $c_P >0$, depending on the degree of $K$, $e$, $v$ and the $j$-invariant of $E$, such that $\widehat{h}_f (\alpha) \geq c$ for all $\alpha \in M_{e}(v) \setminus \PrePer (f)$, and such that there are less than $c_P $ preperiodic points of $f$ in $M_{e}(v)$.
\end{Theorem}

The precise formulation of these results can be found in Theorems \ref{ellipticversion} and \ref{lattesversion}. The paper is organized as follows: In Section \ref{Heights and Lattes maps} we will give a brief introduction to Latt\`es maps and their associated canonical height and we will compare these heights with the N\'eron-Tate height on elliptic curves. In Section \ref{Helpful calculations} we will prove three lemmas which are needed for the proof of our main results. These proofs are given in Section \ref{Main results}. In the last section of this paper we study the behavior of $\widehat{h}_E$ and $\widehat{h}_f$ on $M_{e}^{E}(v)$, resp. $M_e (v)$, when $E$ has any reduction type at $v$.

\vspace{0.2cm}
  
\textit{Acknowledgment:} The author would like to thank Sinnou David for sharing his idea to prove Theorem \ref{ellipticversion} and Sara Checcoli, Walter Gubler, Joseph Silverman, Emmanuel Ullmo and the anonymous referee for helpful comments and suggestions.

In the published version of this article there is an error in the proof of the former Lemma 5.8. This was noticed by Francesco Amoroso and Lea Terracini produced a counterexample to prove that the whole statement was false. I thank both of them for the chance to remove this long standing error.

\end{section}

\begin{section}{Heights and Latt\`es maps}\label{Heights and Lattes maps}

In 1993 Call and Silverman introduced a canonical dynamical height function related to rational maps over $\overline{\mathbb{Q}}$. A very brief introduction to these heights can be given by the following Theorem, which we will treat as a definition of the canonical height.

\begin{Theorem}\label{canonical height}
Let $f\in\overline{\mathbb{Q}}(x)$ be a rational map of degree greater than one. There is a unique height function $\widehat{h}_f$, called the canonical height related to $f$, satisfying
 $$\begin{array}{lcr}
 i) \hspace{0.2cm} \widehat{h}_f (f(\alpha))=\deg(f)\widehat{h}_f (\alpha) & \hspace{0.4cm} \text{ and } \hspace{0.4cm} & ii) \hspace{0.2cm} \widehat{h}_f = h + O(1)
 \end{array}$$
 for all $\alpha \in \overline{\mathbb{Q}}\cup\{\infty\}$, where we set $h(\infty)=0$.
 \end{Theorem}

For a proof of Theorem \ref{canonical height} and additional information on these heights we refer to \cite{Si07}, Chapter 3.4. In this paper we will work with a special class of rational functions.

\begin{Definition}
Let $K$ be a field with characteristic different to $2$ and $3$, and $E$ an elliptic curve over $K$ with given endomorphism $\Psi\neq [0]$ of degree greater than one. Consider a finite covering $\pi : E \rightarrow \mathbb{P}^{1}_{K}$. A map $f$ is called \emph{Latt\`es map associated to $E$} if the diagram 
\begin{align}
\begin{xy}
  \xymatrix{
      E \ar[r]^{\Psi} \ar[d]_\pi    &   E \ar[d]^\pi  \\
      \mathbb{P}_{K}^{1} \ar[r]^f             &   \mathbb{P}_{K}^{1}   
  }
\end{xy} \label{Lattes} 
\end{align}
commutes. If it is necessary to be more precise, we call such a Latt\`es map \emph{associated to $E$, $\pi$ and $\Psi$}.
\end{Definition}

We talk of a Latt\`es map over a field $K$ if it is associated to an elliptic curve over $K$. Let $E$ be given in (dehomogenized) Weierstrass equation with an diagram \eqref{Lattes}, then up to composition with an isogeny $\pi$ is one of the following maps (see \cite{Si07}, Proposition 6.37 and Theorem 6.57). Note that in cases where the $j$-invariant of $E$ is $1728$ or $0$ there are more then one admissible choices for $\pi$.

\begin{equation*}
\pi (x,y)= 
\begin{cases} x & \text{in any case}
\\
x^2 &\text{if $j_E =1728$}
\\
x^3 &\text{if $j_E =0$}
\\
y^2 &\text{if $j_E =0$}
\end{cases}
\end{equation*}

We have the following well known correspondence between the N\'eron-Tate height of an elliptic curve and the canonical height associated to a Latt\`es map. 

\begin{Lemma}\label{heightrelation}
Let $K$ be a subfield of $\overline{\mathbb{Q}}$, $E$ an elliptic curve over $K$ and $f$ a Latt\`es map associated to $E$ with diagram \eqref{Lattes}. Then we have $$\widehat{h}_f \circ \pi = \deg (\pi) \widehat{h}_E \quad ,$$
where $\widehat{h}_f$ is the canonical height to $f$ and $\widehat{h}_E$ is the canonical height on $E$.
\end{Lemma}

Let $E: y^2 = x^3 + Ax +B$ be an elliptic curve defined over a field $K$ of characteristic $\neq 2,3$ and let $\gamma \in \overline{K}^{*}$. The elliptic curve $$E_{\gamma} : y^2 = x^3 + \gamma ^{2} A x + \gamma ^{3}B$$
defined over $K(\gamma)$ is the \textit{twist of $E$ by $\gamma$}. Notice that this is a Weierstrass equation of the curve given by $ y^2 = \gamma(x^3 + Ax +B)$. If $A$, respectively $B$, is equal to zero, then $E_{\gamma}$ is defined over $K(\gamma^3)$, respectively $K(\gamma^2)$. $E$ and $E_{\gamma}$ are isomorphic over $\overline{K}$ and an isomorphism is given by
$$g_{\gamma}: E \tilde{\rightarrow} E_{\gamma} \hspace{0.3cm} ; \hspace{0.3cm} (x,y) \mapsto  (\gamma x,\gamma \sqrt{\gamma}y) \quad .$$
As $g_{\gamma}$ is an isomorphism, it commutes with multiplication by $m \in \mathbb{Z}$. This gives a simple relation between the canonical heights on $E$ and $E_{\gamma}$. For any $P\in E$ we have
\begin{align}
\widehat{h}_E (P) &= \frac{1}{2} \lim_{n\rightarrow\infty}\frac{1}{4^n}h(x([2]^n P)) \nonumber \\  &= \frac{1}{2} \lim_{n\rightarrow\infty}\frac{1}{4^n}h(\gamma^{-1} x([2]^n g_{\gamma}(P))) = \widehat{h}_{E_{\gamma}} (g_{\gamma}(P)). \label{heighttwist}
 \end{align}  

\begin{Proposition}\label{twisting}
Consider a field $K \subseteq \overline{\mathbb{Q}}$. Let $E$ be an elliptic curve defined over $K$ and $f$ a Latt\`es map related to the diagram \eqref{Lattes}. If there is a positive constant $c>0$ such that for every elliptic curve $E'$ defined over $K$, which is $\overline{K}$-isomorphic to $E$, $\widehat{h}_{E'} (P) \geq c$ is true for all $P\in E'(K)\setminus E'_{\rm tor}$, then we have $$\widehat{h}_f (\alpha) \geq \deg(\pi) c \text{ for all } \alpha \in K \setminus \PrePer(f) \quad .$$
In particular this relation is true if $c$ only depends on the $j$-invariant of $E$.
\end{Proposition}

\begin{proof}
The strategy for the proof is the following. Take an arbitrary $\alpha \in K \setminus \PrePer (f)$ and a point $P\in E(\overline{K})$ with $\pi (P) =\alpha$. As $\alpha$ is non-preperiodic, we know that $P$ is not a torsion point. Twist $E$ by a suitable $\gamma$ such that $g_{\gamma}(P) \in E_{\gamma} (K)$ and $E_{\gamma}$ is defined over $K$, then use Lemma \ref{heightrelation}, \eqref{heighttwist} and our assumption to conclude $$\widehat{h}_f (\alpha) = \deg (\pi) \widehat{h}_E (P) = \deg (\pi) \widehat{h}_{E_{\gamma}} (g_{\gamma}(P))\geq \deg (\pi) c \quad .$$

In order to prove the existence of such a $\gamma$ we have to consider four different cases depending on the representation of $\pi$.
We will examine the case $\pi (x,y)=x^2$. Notice that this can only occur if $j_E = 1728$, so we can assume $E: y^2 = x^3 +Ax$. For fixed roots we have $P =(\sqrt{\alpha}, \sqrt[4]{\alpha}\sqrt{\alpha + A})$. Twisting by $\sqrt{\alpha}(\alpha^2 + \alpha A)$ yields $g_{\sqrt{\alpha}(\alpha^2 + \alpha A)}(P)=(\alpha^3 + \alpha^2 A, \alpha^2 + \alpha A) \in E_{\sqrt{\alpha}(\alpha^2 + \alpha A)}(K)$. The elliptic curve $E_{\sqrt{\alpha}(\alpha^2 + \alpha A)}$ is defined over $K$, since $B=0$ and $\sqrt{\alpha}(\alpha^2 + \alpha A)\in K^{\nicefrac{1}{2}}$.

The other cases follow similarly. \end{proof}

\begin{Definition}
Let $p$ be a rational prime and $e,\textsl{f} \in \mathbb{N}$. We call a subfield $L$ of $\overline{\mathbb{Q}}$ \emph{totally $p$-adic of type $(e,\textsl{f})$} if for every $\alpha \in L$ and all $w \in M_{\mathbb{Q}(\alpha)}$, with $w \mid p$, the ramification indices $e_{w\mid p}$ are bounded by $e$ and the residue degrees $\textsl{f}_{w \mid p}$ are bounded by $\textsl{f}$.
\end{Definition}

The next corollary follows immediately from Proposition \ref{twisting} and Theorem 6.3 in \cite{BP05}.

\begin{corollary}\label{bakerpetsche1}
Let $f$ be a Latt\`es map associated to an elliptic curve $E$ over $\mathbb{Q}^{tr}$, with $j$-invariant $j_E$. Then we have
$$\widehat{h}_f (\alpha) \geq \frac{1}{108 (h(j_{E})+10)^5} \text{ for all } \alpha \in \mathbb{Q}^{tr} \setminus \PrePer (f) \quad.$$ 
\end{corollary}

\begin{corollary}\label{bakerpetsche2}
Now let $K$ be a number field, $p$ an odd prime and $E$ an elliptic curve defined over $K$ having no additive reduction at all places of $K$ lying above $p$. If $L / K$ is a totally $p$-adic field of type $(e,\textsl{f})$, for $e,\textsl{f} \in \mathbb{N}$, and $f$ a Latt\`es map associated to $E$ with diagram \eqref{Lattes}, then we have:
\begin{itemize}
\item[i)] $\widehat{h}_f (\alpha) \geq \frac{25}{256}(\frac{\log{p}}{6eM})^3 ( \log(6eM)+\frac{\log{p}}{3e}+\frac{1}{6}h(j_E )+\frac{32}{5})^{-2}$ for all $\alpha \in L \setminus \PrePer (f)$
\item[ii)] $\vert \PrePer (f) \cap L \vert \leq \frac{24eM}{5\log{p}}(\log(6eM)+\frac{\log{p}}{3e}+\frac{1}{6}h(j_E )+\frac{32}{5})+2$,
\end{itemize}
where $M=\max \{p^{6\textsl{f}} +1 +2 p^{3\textsl{f} }, 72 e \nu\}$, with $\nu$ the maximum of $0$ and $-\ord_w(j_E)$ for all places $w \in M_K$ lying above $p$.  
\end{corollary}

\begin{proof} In \cite{BP05}, Theorem 6.7, Baker and Petsche prove effective positive constants $c(e,\textsl{f},p,j_E,K)$ and $c'(e,\textsl{f},p,j_E,K)$ such that $\widehat{h}_E (P)\geq c(e,\textsl{f},p,j_E,K)$ for all $P\in E(L)\setminus E_{\rm tor}$ and $\vert E_{\rm tor}(L)\vert \leq c'(e,\textsl{f},p,j_E,K)$.
In general, the reduction type of an elliptic curve over a place $v$ is not preserved under $\overline{K}$-isomorphisms. So we cannot apply Proposition \ref{twisting} to prove $i)$. By the multiplicativity of the ramification index and the residue degree, every extension of $L$ of degree $n$ is a $p$-adic field of type $(n e,n\textsl{f})$. Let $\alpha\in L$ be arbitrary and $P\in E(\overline{K})$ with $\pi(P)=\alpha$. Then $P$ is defined over a $p$-adic field of type $(6e, 6\textsl{f})$, since the degree of $\pi$ is at most $6$. Moreover, $P$ is a torsion point if and only if $\alpha$ is preperiodic. So, if $\alpha\in L$ is not preperiodic, then $\widehat{h}_f (\alpha)\geq 2c(6e,6\textsl{f},p,j_E,K)$ which is exactly the statement $i)$. By the Riemann-Hurwitz formula, there are at most $2\deg(\pi)$ critical values of $\pi$. Using this, the effective bound $c'(e,\textsl{f},p,j_E,K)$ and the arguments above, a short calculation proves statement $ii)$.
\end{proof}

Notice that the assumption on $E$ to have no additive reduction over the places $v\mid p$ can be achieved for all elliptic curves after a finite extension of the field of definition. In Section \ref{Main results} we will prove that we can drop the necessity of the bound $\textsl{f}$ in Corollary \ref{bakerpetsche2} if $E$ has multiplicative reduction at a finite place $v\mid p$. 

\end{section}

\begin{section}{Helpful calculations}\label{Helpful calculations}

\begin{Lemma}\label{nonempty}
Let $K$ be a number field with non-archimedean valuation $v\mid p$, and let $E$ be an elliptic curve defined over $K$. The set $E(K^{nr,v})\setminus E_{\rm tor}$ is not empty.
\end{Lemma}

\begin{proof}
 Write $E$ in short Weierstrassform $E: y^2 = x^3 + Ax +B$, with $v(A)$, $v(B) \geq 0$. First we consider the case $p\neq 2$. Let $w$ be any valuation on $K^{nr,v}$ which extends $v$. Then, as $v\nmid 2$, for all $\alpha \in K^{nr,v}$ with $w(\alpha)=0$ we also have $\sqrt{\alpha}\in K^{nr,v}$. For all $n\in\mathbb{N}$ set
\begin{equation*}
x_n= 
\begin{cases} np & \text{if $v(B)=0$}
\\
np+\sqrt{A} &\text{if $v(A)=0, v(B) >0$}
\\
np+1 &\text{if $v(A), v(B)>0$}
\end{cases}
\end{equation*}

Then $y_n = \sqrt{x_{n}^{3} + Ax_n +B} \in K^{nr,v}$ for all $n\in\mathbb{N}$. Hence, there are infinitely many points $(x_n, y_n)\in E(K^{nr,v})$ with $[K(x_n, y_n):K]\leq 4$. By Northcott's theorem $E$ has only finitely many torsion points of bounded degree. This proves the lemma in the case $p\neq 2$.

In case where $p=2$ we define $y_n = 2n$ if $v(B)=0$ and $y_n=2n+1$ if $v(B)>0$ for all $n\in\mathbb{N}$. Let $x_n \in \overline{K}$ be such that $(x_n ,y_n) \in E(\overline{K})$ for all $n\in\mathbb{N}$. This means that $x_n$ is a root of $f_n(x)=x^3 + Ax +(B-y_{n}^{2})$. The discriminant of $f_n$ is equal to $\Delta(f_n )=-4A^3 -27(B-y_{n}^{2})^2$. By choice of the $y_n$, $n\in\mathbb{N}$, we have $v(\Delta(f_n))=0$ and hence $x_n \in K^{nr,v}$ for all $n \in \mathbb{N}$. As before we conclude, using Northcott's theorem, that there is a non-torsion point in $E(K^{nr,v})$.
\end{proof}

We will also need a small result concerning ramification indices, which we will state as a lemma.

\begin{Lemma}\label{ramification}
Let $K$ be a field with discrete valuation $v$ and let $L/K$ be a finite and $K'/K$ any field extension. We choose any field which contains $L$ and $K'$ and build the compositum $LK'$ in this field. For all places $w'\mid v$ on $K'L$ define $v'=w'\vert_{K'}$ and $w=w'\vert_{L}$. If the residue field $k(v)$ is perfect, then we have $e_{w'\mid v'}\leq e_{w\mid v}$. %Here, the composita are build in a fixed algebraic closure of $K$.
\end{Lemma}

\begin{proof} Denote by $M$ the maximal unramified extension of $K_v$ inside $L_w$. Then $M/K_v$ is unramified and $L_w /M$ is totally ramified (see \cite{La}, II Proposition 10). Hence, we have $e_{w\mid v}=[L_w :M]$. See for example \cite{Ne}, II Satz 7.2, for the fact that $K'_{v'}M/K'_{v'}$ is also unramified. Thus, we know $e_{w'\mid v'}\leq [(K'L)_{w'}:K'_{v'}M]$. Using the equation $(K'L)_{w'}=L_w K'_{v'}$ we get
$$e_{w\mid v} =[L_w :M] \geq [(K'L)_{w'}:K'_{v'}M] \geq e_{w' \mid v'}$$ as desired. \end{proof}

The real Lambert-$W$ function $W: [-\frac{1}{\mathbf{e}}, \infty ) \rightarrow \mathbb{R}$ is given as the multivalued inverse map of $F(x)=x\mathbf{e}^x$, where $\mathbf{e}=2.71828\dots$ is Euler's number. We have $W(-\frac{1}{\mathbf{e}}) =-1$, but elements in $(-\frac{1}{\mathbf{e}}, 0)$ have two pre-images under $F$. Thus, $W$ has two branches in the interval $[-\frac{1}{\mathbf{e}}, 0)$. The upper branch $W_0 (x)$ tends to $0$ for $x \nearrow	0$ and the lower branch $W_{-1} (x)$ tends to $-\infty$ for $x \nearrow	0$. We do not need deep information on the Lambert-$W$ function and take it mainly as an useful notation. For more information on this function we refer to \cite{CGHJK}.

Similarly as in \cite{BP05} we will use the following lemma.

\begin{Lemma}\label{lambert}
Let $a,b > 0$ be positive constants with $b\geq a$ and let $r: \mathbb{R}_{+} \rightarrow \mathbb{R}$ be given by $r(x)=ax-b-\log{x}$. Then $r(x)$ is positive for all $x > -\frac{1}{a}W_{-1}(-a\mathbf{e}^{-b})$ and we have the inequalities $$\frac{5}{8} < -\frac{1}{a}W_{-1}(-a\mathbf{e}^{-b}) < \frac{8}{5a}(\log{\frac{1}{a}}+b)\quad .$$
\end{Lemma}

\begin{proof} The function $r(x)$ obviously tends to plus infinity, so we have to find the roots of $r(x)$ in order to prove the Lemma. We have
$$\begin{array}{ll}
 & ax-b-\log{x} = 0 \\
 \Leftrightarrow & \mathbf{e}^{-ax} x = \mathbf{e}^{-b} \\
 \Leftrightarrow & -a x \mathbf{e}^{-ax} = -a\mathbf{e}^{-b}\\
 \Leftrightarrow & x \in \left\{ -\frac{1}{a} W_0 (-a\mathbf{e}^{-b}), -\frac{1}{a} W_{-1} (-a\mathbf{e}^{-b}) \right\}
 \end{array}$$
Our assumption on $b$ provides that $W_0 (-a\mathbf{e}^{-b})$ and $W_{-1} (-a\mathbf{e}^{-b})$ are defined. As we have $-\frac{1}{a} W_0 (-a\mathbf{e}^{-b}) \leq -\frac{1}{a} W_{-1} (-a\mathbf{e}^{-b})$, we know that $r(x)\geq 0$ for all $x \geq -\frac{1}{a} W_{-1} (-a\mathbf{e}^{-b})$. As $r(x)$ is strictly increasing on $[-\frac{1}{a} W_{-1} (-a\mathbf{e}^{-b}),\infty)$, this proves the first part of the lemma. Now let $y$ be in the interval $[-\frac{1}{\mathbf{e}},0)$. By definition we have $y=W_{-1}(y)\mathbf{e}^{W_{-1}(y)}$. Multiplying this equation by $-1$ and taking the logarithm yields
\begin{align*} \log(-y) &= \log(-W_{-1}(y)) + W_{-1}(y) \\ &= W_{-1}(y)\left( 1-\frac{\log(-W_{-1}(y))}{-W_{-1}(y)}\right) \leq \frac{\mathbf{e}-1}{\mathbf{e}}W_{-1}(y)  .\end{align*}
As $-W_{-1}(y) \geq 1$, this leads to the inequality
$$W_{-1}(y) \leq \log(-y) \leq \frac{\mathbf{e}-1}{\mathbf{e}}W_{-1}(y) \quad . $$
Applying this to $-\frac{1}{a} W_{-1} (-a\mathbf{e}^{-b})$ and using $b\geq a$ gives us $$\frac{\mathbf{e}}{(\mathbf{e}-1)a}(\log{\frac{1}{a}}+b) \geq -\frac{1}{a} W_{-1} (-a\mathbf{e}^{-b}) \geq 1 - \frac{\log{a}}{a} \geq \frac{\mathbf{e}-1}{\mathbf{e}}\quad .$$
The estimation $\frac{\mathbf{e}}{\mathbf{e}-1}<\frac{8}{5}$ concludes the proof. \end{proof}
 
\end{section}

\begin{section}{Proof of the main results}\label{Main results}

From now on we fix the following notations. Let $K$ be a number field with non-archimedean absolute value $v \mid p$ and $E$ an elliptic curve over $K$ with $j$-invariant $j$. By $d$ we denote the degree $[K:\mathbb{Q}]$ and by $d_v$ the local degree $[K_v : \mathbb{Q}_p ]$. Let further $f$ be a Latt\`es map associated to $E$. 
For a fixed $e \in \mathbb{N}$ we define $$M_e (v) :=\{\alpha \in \overline{\mathbb{Q}} \vert e_{w\mid v} \leq e \text{ for all } w \in M_{K(\alpha )},~w\mid v \} \text{ and }$$ $$M_{e}^{E}(v) :=\{ P \in E (\overline{\mathbb{Q}}) \vert e_{w\mid v} \leq e \text{ for all } w \in M_{K(P)},~w\mid v \},$$ where $e_{w \mid v}$ is the ramification index of $w$ over $v$. Let further $\widehat{h}_f$ be the canonical height related to $f$ and $\widehat{h}_E$ the N\'eron-Tate height on $E$. Based on an idea of Sinnou David we will proof:

\begin{Theorem}\label{ellipticversion}
If $E$ has split-multiplicative reduction at $v\mid p$, then there are effective computable constants $c'(j,d,e,v)$, $c'_T (j,d,e,v)>0$ only depending on $j$, $d$, $e$ and $v$, with $\widehat{h}_E (P) \geq c'(j,d,e,v)$ for all $P \in M_{e}^{E}(v) \setminus E_{\rm tor}$ and such that there are less than $c'_T (j,d,e,v)$ torsion points in $M_{e}^{E}(v)$. More precisely, we have
 \begin{itemize}
 \item[i)] $\widehat{h}_E (P) \geq \frac{\frac{\log{p}}{2d}\mathfrak{c} - 3\log{2}}{(8\mathfrak{c}^3 - 2\mathfrak{c})(e!\ord_{v}(j^{-1}))^2}>0 \text{ for all } P\in M_{e}^{E} (v)\setminus E_{\rm tor}$
 \item[ii)] $\vert E_{\rm tor}\cap M_{e}^{E}(v) \vert < \frac{1}{3}\left( e!\mathfrak{c}\ord_v (j^{-1})\right)^3 + \frac{1}{2}\left( e!\mathfrak{c}\ord_v (j^{-1})\right)^2 ,$
 \end{itemize}
where $\mathfrak{c}:=\left\lceil \frac{10d}{\log{p}}\left( \log(\frac{6d}{\log{p}})+\frac{1}{6}h(j)+\frac{32}{5}\right)\right\rceil$.
\end{Theorem}
  
Note that since $E$ has multiplicative reduction at $v\mid p$, $h(j)$ is at least $\frac{\log{p}}{d}$. Hence, $\mathfrak{c}\geq 1$. 
We will also prove the dynamical analogue, which states the following.

\begin{Theorem}\label{lattesversion}
If $E$ has split-multiplicative reduction at $v$, then there are effective computable constants $c(j,d,e,v)$, $c_P (j,d,e,v)>0$ only depending on $j$, $d$, $e$ and $v$, with $\widehat{h}_f (\alpha) \geq c(j,d,e,v)$ for all $\alpha \in M_e \setminus \PrePer (f)$ and such that there are less than $c_P (j,d,e,v)$ preperiodic points in $M_e (v)$. With the notation of Theorem \ref{ellipticversion}, we have
\begin{itemize}
\item[i)] $\widehat{h}_f (\alpha)\geq \frac{\frac{\log{p}}{2d}\mathfrak{c} - 3\log{2}}{2(8\mathfrak{c}^3 - 2\mathfrak{c})(e!\ord_{v}(j^{-1}))^2}>0 \text{ for all }\alpha\in M_e (v) \setminus \PrePer(f)$
\item[ii)] $\vert \PrePer(f) \cap M_{e} (v) \vert < \frac{4}{3}\left( e! \mathfrak{c}\ord_v (j^{-1}) \right)^3 +\left( e!\mathfrak{c}\ord_v (j^{-1}) \right)^2.$
\end{itemize}
\end{Theorem}
  
\begin{proof}[Proof of Theorem \ref{ellipticversion}] 
Let $P$ be a point in $M_{e}^{E} (v)$, $w\mid v$ a valuation on $K(P)$ and $k_P (w)$ the residue field of $K(P)_w$. We choose a minimal Weierstrass equation for $E$ over $K(P)_w$ with discriminant $\Delta$. Then $\widetilde{E}$ is the reduction of $E$ modulo $w$ and $\widetilde{E}_{ns}$ is the set of all non-singular points in $\widetilde{E}$. We set $E_{0}(K(P)_w):=\{ P\in E (K(P)_w) \vert \widetilde{P}\in\widetilde{E}(k_P (w))_{ns}\}$. $E_{0}(K(P)_w)$ is a subgroup of $E (K(P)_w)$ of index $\ord_w (j^{-1})=e_{w\mid v}\ord_v (j^{-1})$ (see \cite{Si99}, Cor. IV.9.2). So we have $e_{w\mid v}\ord_v (j^{-1}) P\in E_{0}(K(P)_w)$. From the choice of $P$ it is clear that $Q:=e!\ord_v (j^{-1}) P \in E_{0}(K(P)_w)$ for all $w\mid v$. Recall that $w(.)$ is the unique valuation on $K(P)$ extending the usual $p$-adic valuation on $\mathbb{Q}$, and $\ord_w (.)=\frac{e_{w\mid p}}{\log{p}} w(.)$ is the normalization of $w(.)$ to a function onto $\mathbb{Z}$.

We take the local heights $\lambda_w$ on $E(K(P)_w )\setminus{0}$ normalized such that we have the equation
$$ \widehat{h}_E (P) =\frac{1}{[K(P):\mathbb{Q}]}\sum_{w \in M_{K(P)}} d_w \lambda_w (P) \quad \forall P \in E (\overline{K})\setminus{0} \quad .  $$

For $Q \in E_{0}(K(P)_w)$ and $w\mid v$, $w \in M_{K(Q)}$ we have 
\begin{equation}\lambda_w (Q) = \frac{1}{2} \max \{w(x(Q)),0\}+\frac{1}{12}w(\Delta) \geq \frac{1}{12}w(j^{-1})=\frac{1}{12}v(j^{-1}) \quad . \label{ueberv}\end{equation}
See \cite{Si99}, Theorem VI.4.1. Notice that $E$ has split-multiplicative reduction over $K(P)$ at every $w\mid v$. Hence, we can use $w(\Delta)=w(j^{-1})$.

We define the set $\Lambda_s =\{iQ \vert i \in \mathbb{N}, i \leq s \}$, for all $s\in \mathbb{N}$, such that $\Lambda_s$ consists of exactly $s$ points. Now we will estimate $\widehat{h}_E (P)$, respectively $\widehat{h}_E (Q)$, using bounds for the local heights. For a given absolute value $w$ we set $w^+$ to be the maximum of $w$ and $0$.

If $w$ is archimedean, then we can use a theorem of Elkies improved by Baker and Petsche (see \cite{BP05}, Appendix A). Namely
\begin{equation}
\sum_{\substack{R,R' \in \Lambda_s \\ R\neq R'}} \lambda_w (R-R') \geq  -\frac{s}{2}\log{s}-\frac{16}{5}s - \frac{1}{12}w^+ (j^{-1})s . \label{archimedisch}
\end{equation}
For a non-archimedean $w\in M_{K(P)}$, with $w(j^{-1})\leq 0$, $E$ has potential good reduction at $w$ (See for example \cite{Si86}, Proposition VII.5.1). Let $K' / K(P)$ be a finite extension such that $E$ has good reduction over $K'$ at a $w' \mid w$. Then the equation we have used in \eqref{ueberv} shows that $\lambda_{w'}$ is not negative on $E (K'_{w'})$. As $\lambda_{w'}$ and $\lambda_{w}$ coincide on $E (K(P)_w)$, $\lambda_{w}$ is a non-negative function.

For non-archimedean absolute values $w$ with $w(j^{-1})> 0$ \cite{HS90}, Proposition 1.2, gives the inequality
$$ \sum_{\substack{R,R' \in \Lambda_s \\ R\neq R'}} \lambda_w (R-R') \geq \frac{1}{12}\left( \frac{s}{\ord_w (j^{-1})}\right)^2 w(j^{-1}) -\frac{s}{12}w(j^{-1}) .$$

Thus, for an arbitrary non-archimedean absolute value we will use the estimation
\begin{equation}
 \sum_{\substack{R,R' \in \Lambda_s \\ R\neq R'}} \lambda_w (R-R') \geq -\frac{s}{12}w^+ (j^{-1}) . \label{nichtarchimedisch}
\end{equation}

With \eqref{ueberv}, \eqref{archimedisch}, \eqref{nichtarchimedisch} we get:
\begin{align*} &\sum_{\substack{R,R' \in \Lambda_s \\ R\neq R'}} \widehat{h}_E (R-R')=  \sum_{\substack{R,R' \in \Lambda_s \\ R\neq R'}} \frac{1}{[K(P):\mathbb{Q}]} \sum_{w\in M_{K(P)}} d_w \lambda_w (R-R') \\ 
\geq &\frac{1}{[K(P):\mathbb{Q}]}\sum_{w\mid \infty} d_w (-\frac{1}{2}s\log{s} - \frac{16}{5}s)  - \frac{1}{[K(P):\mathbb{Q}]}\sum_{w\mid \infty} d_w w^+ (j^{-1})\frac{1}{12}s \\
- &\frac{1}{[K(P):\mathbb{Q}]} \sum_{w \nmid \infty, w \nmid v} d_w \frac{s}{12}w^+ (j^{-1}) 
+ \sum_{\substack{R,R' \in \Lambda_s \\ R\neq R'}}  \frac{1}{[K(P):\mathbb{Q}]} \sum_{w\mid v} d_w \frac{1}{12}v(j^{-1}) \end{align*}
The last sum consists of exactly $s^2 -s$ terms. Hence, the above is equal to
\begin{equation*} -\frac{1}{2}s\log{s}-\frac{16}{5}s- \frac{s}{[K(P):\mathbb{Q}]}\sum_{w\nmid v} \frac{d_w w^+ (j^{-1})}{12} + \frac{s^2 -s}{[K(P):\mathbb{Q}]} \sum_{w\mid v} \frac{d_w v(j^{-1})}{12}. \end{equation*}

We know that for all $w\mid v$ we have $v(j^{-1})=w(j^{-1})=w^+ (j^{-1})$. Thus, we can use the definition of the standard logarithmic height $h$ to obtain 
\begin{align} \sum_{\substack{R,R' \in \Lambda_s \\ R\neq R'}} \widehat{h}_E (R-R')  &\geq \frac{d_v v(j^{-1})}{12d}s^2 - \left(\frac{1}{12} h(j) + \frac{16}{5}\right) s - \frac{1}{2}s\log{s} \nonumber\\
  &\geq \frac{\log{p}}{12d}s^2 - \left(\frac{1}{12} h(j) + \frac{16}{5}\right) s - \frac{1}{2}s\log{s} \label{hoehe1} . \end{align}

If $P$ is a torsion point, then the left hand side is equal to zero. Clearly $h(j)$ is greater than or equal to $\frac{d_v v(j^{-1})}{d} \geq \frac{\log{p}}{d}$, so we can apply Lemma \ref{lambert} to deduce that the right hand side is greater than zero for $$s\geq \frac{48d}{5\log{p}}\left(\log(\frac{6d}{\log{p}})+\frac{1}{6}h(j)+\frac{32}{5}\right).$$ As $s$ is a natural number, we get a contradiction for $s=\mathfrak{c}$. This shows that there cannot exist a torsion point $P \in M_{e}^{E}(v)$ such that the order of $e!\ord_v (j^{-1})P$ is greater than or equal to $\mathfrak{c}$. Hence, there cannot exist a torsion point $P \in M_{e}^{E} (v)$ of order greater than or equal to $\mathfrak{c} \ord_v (j^{-1}) (e!)$. Using $\vert E [k] \vert = k^2$ and $0 \in E[k]$, for all $k\in \mathbb{N}$, we get that there are less than $\frac{1}{3}\left( e!\mathfrak{c}\ord_v (j^{-1}) \right)^3 +\frac{1}{2}\left( e!\mathfrak{c}\ord_v (j^{-1}) \right)^2 $ torsion points in $M_{e}^{E}(v)$.

From now on we assume that $P$ is no torsion point. Then $\Lambda_s$ is defined for all $s \in\mathbb{N}$ and so \eqref{hoehe1} is valid for all $s \in \mathbb{N}$.
The definition of $\Lambda_s$ and the property $\widehat{h}_E (kQ)=k^2 \widehat{h}_E (Q)$ for all $k\in\mathbb{Z}$ leads us to 
\begin{equation} \sum_{\substack{R,R' \in \Lambda_s \\ R\neq R'}} \widehat{h}_E (R-R') = \left(2\sum_{i=1}^{s-1} i^2 (s-i)\right) \widehat{h}_E (Q) = \left(\frac{1}{6}s^{4}-\frac{1}{6} s^2 \right) \widehat{h}_E (Q).\label{hoehe2} \end{equation}

If we further use \eqref{hoehe1} and the definition of $Q$, we find that the height $\widehat{h}_E (P)$ is bounded from below by
$$C'(j,d,e,v):=\max_{s\in\mathbb{N}}\frac{\frac{\log{p}}{2d}s - \left(\frac{1}{2} h(j) + \frac{96}{5}\right) - 3\log{s}}{(s^{3}- s)(e!\ord_v (j^{-1}))^2}.$$
$C'(j,d,e,v)$ is obviously positive. In what follows we will give a lower bound for $C'(j,d,e,v)$. Let $\mathfrak{c}_W := -\frac{6d}{\log{p}}W_{-1}\left(-\frac{\log{p}}{6d}H(j)^{\nicefrac{1}{6}}\mathbf{e}^{\nicefrac{32}{5}}\right)$, where $H(j)$ is the multiplicative height of $j$, be the greatest root of the real function $r(x)=\frac{\log{p}}{2d}x -\left(\frac{1}{2}h(j)+\frac{96}{5}\right)-3\log{x}$ (see Lemma \ref{lambert}). Then we know that this function is strictly positive for all $x > \mathfrak{c}_W$. In particular, we have $r(2x)\geq \frac{\log{p}}{2d}x - 3\log{2}$ for all $x\geq\mathfrak{c}_W$, with equality if and only if $x=\mathfrak{c}_W$. Again by Lemma \ref{lambert} we have $1<2\mathfrak{c}_W<2\mathfrak{c}$. With this we finally deduce
\begin{equation}
C'(j,d,e,v)\geq \frac{r(2\mathfrak{c})}{(8\mathfrak{c}^3 - 2\mathfrak{c})(e!\ord_{v}(j^{-1}))^2}\geq \frac{\frac{\log{p}}{2d}\mathfrak{c} - 3\log{2}}{(8\mathfrak{c}^3 - 2\mathfrak{c})(e!\ord_{v}(j^{-1}))^2} >0 ,\nonumber \end{equation}
which concludes the proof. \end{proof}

\begin{proof}[Proof of Theorem \ref{lattesversion}] We will combine Lemma \ref{heightrelation} and Theorem \ref{ellipticversion}. $E$ is assumed to have split-multiplicative reduction at $v$, and hence $\pi$ has degree two. Let $\alpha$ be a point in $M_e (v)$ and take a point $P \in E(\overline{K})$ with $\pi(P)=\alpha$. Then for all $w\mid v$ in $M_{K(P)}$ we have either $e_{w\mid v} \leq e$ or $e_{w\mid v} =2e'$ with $e'\leq e$. Now we can start exactly the proof above with $Q:=e!2\ord_v(j^{-1})P$ instead of $e!\ord_v(j^{-1})P$. 

If $\alpha$ is preperiodic, then $P$ is a point of order less than $e!2\mathfrak{c}\ord_v (j^{-1})$. So there are less than $\frac{8}{3}\left( e!\mathfrak{c}\ord_v (j^{-1}) \right)^3 + 2 \left( e! \mathfrak{c}\ord_v (j^{-1}) \right)^2 $ choices for $P$. 
As every $\alpha$, that is no critical value of $\pi$, has exactly two pre-images under $\pi$, and there are exactly $4$ critical values of $\pi$ (see \cite{Si07}, Lemma 6.38), we get
$$\vert \PrePer(f) \cap M_{e} (v) \vert < \frac{4}{3} \left( e! \mathfrak{c}\ord_v (j^{-1}) \right)^3 + \left( e! \mathfrak{c}\ord_v (j^{-1}) \right) +2 .$$
If $\alpha$ is no preperiodic point, then $P$ is not a torsion point and we have
$$\widehat{h}_f (\alpha)=2\widehat{h}_E (P) \geq \frac{\frac{\log{p}}{2d}\mathfrak{c} - 3\log{2}}{2(8\mathfrak{c}^3 - 2\mathfrak{c})(e!\ord_{v}(j^{-1}))^2}>0 \quad .$$
This concludes the proof.
\end{proof}

\end{section}

\begin{section}{Corollaries and additional results}\label{Corollaries}

Now we want to study the behavior of the canonical height $\widehat{h}_E$ on the set $M_{e}^{E} (v)$ if $E$ has not split-multiplicative reduction at $v$. If $E$ has multiplicative or potential multiplicative reduction at $v$, then it has split-multiplicative reduction after a finite field extension. So Theorem \ref{ellipticversion} will be true in this case after a small adjustment of the constants. In the case of good reduction of $E$ at $v$ the criterion of N\'eron-Ogg-Shafarevich will show that there are infinitely many torsion points and points of arbitrary small positive height in $M_{e}^{E}(v)$ for all $e\in\mathbb{N}$. If $E$ has additive potential good reduction at $v$, then we can prove that Theorem \ref{ellipticversion} is true for $e=1$, i.e. for the field $K^{nr,v}$.

\begin{corollary}\label{additivered}
Let $c'(j,d,e,v)$ and $c'_T (j,d,e,v)$ be the constants given in Theorem \ref{ellipticversion}. If $E$ has nonsplit-multiplicative or additive potential multiplicative reduction at $v\mid p$, then we have
  \begin{itemize}
 \item[i)] $\widehat{h}_E (P) \geq c'(j,kd,e)\text{ for all } P\in M_{e}^{E} (v)\setminus E_{\rm tor}$
 \item[ii)] $\vert E_{\rm tor}\cap M_{e}^{E}(v) \vert < c'_T (j,kd,e).$
 \end{itemize}
 Here $k\leq48$ is the smallest degree of a field extension $K'/K$ such that $E$ over $K'$ has split-multiplicative reduction at a place $v'\mid v$, $v' \in M_{K'}$.
\end{corollary}

\begin{proof} By assumption there is a field extension $K'/K$ such that $E$ over $K'$ has split-multiplicative reduction at a place $v' \mid v$ of $K'$. Assume furthermore that the degree of $K'/K$ is minimal with this property and denote this degree by $k$. Following the proofs of Prop. VII 5.4 c) and Cor. A.1.4 a) in \cite{Si86} we find that $k\leq 48$.

Let $P$ be a point in $M_{e}^{E} (v)$ and $w' \mid v'$ an extension of $v'$ to $K'(P)$. Denote the restriction of $w'$ to $K(P)$ with $w$. Then by assumption and Lemma \ref{ramification} we have $e\geq e_{w\mid v} \geq e_{w' \mid v'}$.
Thus, we can apply Theorem \ref{ellipticversion} with $d$ replaced by $kd$. \end{proof}

We have proven a lower bound for the canonical heights $\widehat{h}_E$ and $\widehat{h}_f$ on sets. Our main interest concerns lower bounds on fields. For a field $L$ lying inside $M_e$ for some $e\in \mathbb{N}$, Theorem \ref{ellipticversion} and Theorem \ref{lattesversion} give us lower bounds for the canonical heights on $E (L)\setminus E_{\rm tor}$ and $L\setminus \PrePer(f)$. But we can achieve much nicer bounds if we additionally assume that $L / K$ is normal. In this case the term $e!$ in our bound can be replaced by $e$. Formally:

\begin{corollary}\label{galoiscase}
Let $E$ be an elliptic curve over $K$ with split-multiplicative reduction at $v\mid p$ and let $\mathfrak{c}$ be as in Theorem \ref{ellipticversion}. Let further $L / K$ be a Galois extension with $L\subset M_e (v)$, for a fixed $e\in\mathbb{N}$. Then we have
 \begin{itemize}
 \item[i)] $\widehat{h}_E (P) \geq \frac{\frac{\log{p}}{2d}\mathfrak{c} - 3\log{2}}{(8\mathfrak{c}^3 - 2\mathfrak{c})(e\ord_{v}(j^{-1}))^2}>0 \text{ for all } P\in E (L) \setminus E_{\rm tor}$
 \item[ii)] $\vert E_{\rm tor}(L) \vert < (\mathfrak{c} \ord_v (j^{-1}) e)^2 .$
 \end{itemize}
\end{corollary}

\begin{proof} In the proof of Theorem \ref{ellipticversion}, $e!$ was an upper bound for the lowest common multiple of all $e_{w\mid v}$, $w\in M_{K(P)}$, which does not depend on $P$. Now let $K'(P)$ be the normal closure of $K(P)$. From our assumption we know that $K'(P)$ is contained in $L \subset M_e (v)$. Thus, we have $e_{w'\mid v}=e_v \leq e$ for all $w' \in M_{K'(P)}$ lying above $v$. Exactly as in the proof of Theorem \ref{ellipticversion} we achieve that there is no torsion point in $E(L)$ of order greater than $\mathfrak{c} \ord_v (j^{-1}) e$. Let $P$ be a torsion point in $E(L)$ of maximal order $k\leq \mathfrak{c} \ord_v (j^{-1}) e$. We claim that all torsion points in $E(L)$ lie in $E[k]$. Assume this is not the case, then there exists a torsion point $P'\in E(L)$ of order not dividing $k$. The order of the point $P + P' \in E(L)$ is exactly the smallest common multiple of the orders of $P$ and $P'$, and hence it is greater than $k$. This is a contradiction to the maximality of $k$, proving the claim. We can conclude that there are less than $k^2 \leq (\mathfrak{c} \ord_v (j^{-1}) e)^2$ torsion points in $E(L)$.

Statement $i)$ follows exactly as in the proof of Theorem \ref{ellipticversion}.  
\end{proof}

\begin{Remark}
Obviously Corollary \ref{galoiscase} similarly holds for places $v$ where $E$ has potential multiplicative reduction. Moreover, both corollaries have a dynamical analogue for Latt\`es maps associated to $E$. This can be seen by combining the respective proof with the proof of Theorem \ref{lattesversion}. 
\end{Remark}

In the case where $E$ has additive reduction at $v$ and $e=1$ we can also use our computation from Theorem \ref{ellipticversion}. The following proposition is actually a remark of Joseph Silverman in an email to the author.

\begin{Proposition}\label{unramified}
Let $v\mid p$ be a finite place of $K$ where $E$ has additive reduction. With $\mathfrak{c}$ as in Theorem \ref{ellipticversion}, we have
\begin{itemize}
 \item[i)] $\widehat{h}_E (P) \geq \frac{\frac{\log{p}}{2d}(\mathfrak{c}+2) - 3\log{2}}{(8(\mathfrak{c}+2)^3 - 2(\mathfrak{c}+2))144}>0 \text{ for all } P\in M_{1}^{E} (v)\setminus E_{\rm tor}$
 \item[ii)] $\vert E_{\rm tor} \cap M_{1}^{E}(v) \vert < (12\mathfrak{c}+24)^2.$
 \end{itemize}
\end{Proposition}

We will prove below that $\mathfrak{c}\geq -1$ in the setting of Proposition \ref{unramified} (which is not quite obvious from the definition).

\begin{proof} The proof is almost the same as in the case of split-multiplicative reduction. Let $P$ be in $M_{1}^{E} (v)$ and $w\mid v$ a place of $K(P)$. By assumption $v$ is unramified in the extension $K(P)/K$. Hence, $E$ has still additive reduction at $w$ over $K(P)$. Fix a minimal Weierstrass equation of $E$ over $K(P)_w$, and denote the discriminant of $E$ by $\Delta$. Then $E_0 (K(P)_w)$ is a subgroup of $E(K(P)_w)$ of order at most $4$ (see \cite{Si99}, Corollary IV 9.2). Thus, we have $Q:= 12P \in E_0 (K(P)_w)$ for all $w\mid v$ on $K(P)$. The explicit formula for the local height $\lambda_w$ on $E_0 (K(P)_w)$ gives us
$$\lambda_w (Q) \geq \frac{1}{12}w(\Delta) \geq \frac{1}{12e_{v\mid p}}\log{p} > 0 ~\text{ for all } w\mid v \text{ on } K(P) \quad .$$
This follows from the fact that $w(\Delta)=0$ if and only if $E$ has good reduction at $w$ (\cite{Si86}, Proposition VII.5.1), which is not the case as we have noticed above.
Define the set $\Lambda_s := \{iQ \vert i \in \mathbb{N}, i \leq s\}$ for all $s \in \mathbb{N}$ such that $\Lambda_s$ consists of exactly $s$ points. As before we get the lower bound
$$\sum_{\substack{R,R' \in \Lambda_s \\ R\neq R'}} \widehat{h}_E (R-R') \geq \frac{\log{p}}{12d}s^2 - \left(\frac{1}{12} h(j) + \frac{\log{p}}{12d}+\frac{16}{5}\right) s - \frac{1}{2}s\log{s} \quad .$$
Using Lemma \ref{lambert} and the definition of $\mathfrak{c}$ we find that the right hand side is greater than $0$ for $s$ greater than $\mathfrak{c}+\frac{8}{5}< \mathfrak{c}+2$. Lemma \ref{lambert} also tells us $\mathfrak{c}+\frac{8}{5} > \frac{5}{8}$ which implies $\mathfrak{c}+2\geq 1$. Thus, there cannot exist a torsion point of order greater than $12\mathfrak{c}+24$. Notice that $M_{1}^{E}(v)=E(K^{nr,v})$, and therefore $M_{1}^{E}(v)$ is an abelian group. As in the proof of Corolarry \ref{galoiscase}, we conclude that all torsion points in $M_{1}^{E}$ lie in the set $E[k]$, for a $k\leq 12\mathfrak{c}+24$. This implies part $ii)$.
If $P$ is not a torsion point we conclude
$$\widehat{h}_E (P) \geq \frac{\frac{\log{p}}{2d}(\mathfrak{c}+2) - 3\log{2}}{(8(\mathfrak{c}+2)^3 - 2(\mathfrak{c}+2))144}>0 \quad .$$
This again follows exactly as in the proof of Theorem \ref{ellipticversion}.
\end{proof}

Our results show that $K^{nr,v}$ has the Bogomolov property relative to $\widehat{h}_E$ if $E$ has bad reduction at $v$. In order to construct other examples of fields lying inside some $M_e (v)$ the following Lemma might be helpful.
 
\begin{Lemma}\label{compositum}
Let $L/K$ and $M/K$ be field extensions such that $L \subset M_{e} (v)$ and $M\subset M_{e'}(v)$ for a non-archimedean absolute value $v$ on $K$. Then $LM \subset M_{ee'}(v)$.
\end{Lemma}

\begin{proof} Let $F\subset M$ be any subfield with $[F:K]\leq \infty$. Moreover we choose any $w' \in M_{LM}$, with $w'\mid v$, and define $w=w'\vert_{L}$ and $v'=w'\vert_F$. With Lemma \ref{ramification} we conclude $e_{w'\mid v'}\leq e_{w\mid v} \leq e$. This leads us to $e_{w'\mid v}=e_{w'\mid v'}e_{v'\mid v}\leq e e'$. The fact that for every $\alpha\in LM$ there exists such a finite extension $F/K$ with $\alpha \in LF$ concludes the proof. 
\end{proof}

\begin{Proposition}\label{goodred}
If $E$ has good reduction at $v\mid p$ and $f$ is an associated Latt\`es map, then neither of the statements in Theorem \ref{ellipticversion} and \ref{lattesversion} is true.
\end{Proposition}

\begin{proof} Lets start with Theorem \ref{ellipticversion}. By the criterion of N\'eron-Ogg-Sha\-fare\-vich all points of order $m$ with $p\nmid m$ are unramified above $v$ (see \cite{Si86}, Theorem VII 7.1). In particular, there are infinitely many torsion points in $E(K^{nr,v})$ and hence in $M_{e}^{E}(v)$ for all $e \in \mathbb{N}$. 

Take an arbitrary point $P_0 \in E(K^{nr,v}) \setminus E_{\rm tor}$ and let $\{P_n\}_{n \in \mathbb{N}}$ be a sequence of points in $E(\overline{K})$, satisfying $[m]P_n=P_{n-1}$. This is possible by Lemma \ref{nonempty}. By \cite{Si86}, Proposition VIII 1.5 b), all $P_n$ are in $E(K^{nr,v}) \setminus E_{\rm tor}$ and we have
$$\widehat{h}_E (P_n)=\frac{1}{m^2} \widehat{h}_E (P_{n-1})= \frac{1}{m^{2n}} \widehat{h}_E (P_{0}) \rightarrow 0 \quad .$$
Thus, there are points of arbitrarily small height in $M_{e}^{E}(v)$ for all $e \in \mathbb{N}$.

Contradictions for the statements in Theorem \ref{lattesversion} follow quickly. If $P$ is in $E(K^{nr,v})$, then $\pi(P)$ is in $K^{nr,v}$. The degree of $\pi$ is finite so the equation $\PrePer(f)=\pi(E_{\rm tor})$ shows that there are infinitely many preperiodic points of $f$ in $K^{nr,v}$. Also with Lemma \ref{heightrelation} and the above we conclude that $\widehat{h}_f$ can get arbitrary small on $K^{nr,v}\setminus \PrePer(f)$. \end{proof}

\begin{Example}\label{counterexample elliptic}\rm
As usual let $K$ be a number field, $E: y^2 = x^3 + Ax + B$ an elliptic curve defined over $K$, and let $f$ be a Latt\`es map associated to $E$ and $\pi(x,y)=x$. The following example shows two things. Firstly, the Bogomolov property relative to $\widehat{h}_E$ is in general not preserved under finite field extensions. Secondly, a field can have the Bogomolov property relative to $\widehat{h}_E$ but not relative to $\widehat{h}_f$.

We will use the theory of twists of an elliptic curve.
Assume that $E$ has additive reduction at a finite $v\in M_K$, and that there is an element $\gamma \in K$ such that the twist $E_\gamma$ has good reduction at $v \in M_K$. (One might choose $K=\mathbb{Q}$, $v=p\geq 3$, $E: y^2 = x^3 + p^2 x$ and $\gamma= p^{-1}$). Let $g_{\gamma^{-1}}: E_{\gamma} \rightarrow E$ be the isomorphism from  \eqref{heighttwist}. Notice that $\sqrt{\gamma}$ cannot be an element in $K^{nr,v}$, since the reduction type of $E$ at $v$ changes if we extend $K$ to $K(\sqrt{\gamma})$.
As $E_{\gamma}$ has good reduction at $v$, Proposition \ref{goodred} yields a sequence $\{ P_i \}_{i \in \mathbb{N}}$ in $E_{\gamma}(K^{nr,v})$ such that $\widehat{h}_{E_{\gamma}} (P_i) \rightarrow 0$, as $i \rightarrow \infty$. The elements $g_{\gamma^{-1}}(P_i)$ all lie in $E(K^{nr,v}(\sqrt{\gamma}))$, and from the definition of $g_{\gamma^{-1}}$ we know that $\alpha_i :=\pi(g_{\gamma^{-1}}(P_i ))$ is in $K^{nr,v}$. Using Lemma \ref{heightrelation} and \eqref{heighttwist} we get
$$\widehat{h}_f (\underbrace{\alpha_i}_{\in K^{nr,v}}) = 2 \widehat{h}_E (\underbrace{g_{\gamma^{-1}}(P_i)}_{\in K^{nr,v}(\sqrt{\gamma})}) = 2 \widehat{h}_{E_{\gamma}}(P_i) \rightarrow 0 \quad ,$$
as $i \rightarrow \infty$. But $K^{nr,v}$ has the Bogomolov property relative to $\widehat{h}_{E}$, as $E$ has bad reduction at $v$ (see Proposition \ref{unramified}).
\end{Example}

THE FOLLOWING IS NOT EQUIVALENT TO THE CONTENT OF THE PUBLISHED VERSION!

The examples from above are the only known examples, where the Bogomolov property with respect to $\widehat{h}_E$ is not preserved under finite field extensions. In contrast to the statement given in the published article, I am  not able to classify all these field extensions. This is due the fact that the claimed Lemma 5.8 in the previous version of this article is not true - a counterexample was found by Francesco Amoroso and Lea Terracini.

 All I can say is the trivial consequence of the preceding results:

\begin{Proposition}\label{thm:finex}
 Let $E$ be an elliptic curve defined over the number field $K$ with non-archimedean absolute value $v$, and let $L/K$ be a finite Galois extension. Then $L\cdot K^{nr,v}$ has the Bogomolov property relative to $\widehat{h}_E$ if  there is a $w\in M_L$, $w\mid v$, such that $E/L$ has bad reduction at $w$.
\end{Proposition}

\end{section}

\renewcommand{\thefootnote}{}

\footnote{Lukas Pottmeyer, ...}

\end{document}